\pdfoutput=1

\documentclass[12pt]{article}
\usepackage[utf8]{inputenc}
\usepackage{amsmath}
\usepackage{amsthm}
\usepackage{amscd}
\usepackage{comment}
\usepackage[english]{babel}
\usepackage{tikz}
\usepackage{color,graphicx,amssymb,tikz-cd,wrapfig}
\usepackage{amsmath,amsthm,amsfonts,amssymb,stmaryrd}
\usepackage{bm}
\usepackage[rightcaption]{sidecap}
\usepackage{geometry}
\usepackage{hyperref}
\usepackage[figurename=Fig.]{caption}

\usepackage[T1]{fontenc}
\usepackage{combelow}
\usepackage{newunicodechar}
\newunicodechar{ț}{\cb{t}}

\usepackage{lipsum}   % For generating dummy text

\usepackage{xcolor}

\usepackage{enumerate}
\usepackage{pb-diagram}
\usepackage{graphicx} %package to manage images
\usepackage[all]{xy}
\usepackage{appendix}
\usetikzlibrary{intersections,calc,arrows.meta}

\title{Non-geometric property (T) of warped cones}

\author{Jintao Deng\thanks{Department of Mathematics, State University of New York at Buffalo, Buffalo, NY, 14260} \and Ryo Toyota\thanks{Department of Mathematics, Texas A\&M University, College Station, TX 77843, USA}  }%
\date{ }

\newtheorem{thm}{Theorem}[section]
\newtheorem{dfn}[thm]{Definition}
\newtheorem{lem}[thm]{Lemma}

\newtheorem{rmk}[thm]{Remark}

\newtheorem{prop}[thm]{Proposition}

\newtheorem{qustn}[thm]{Question}

\newcommand{\HH}{{\mathcal H}}

\newcommand{\ControlledSupport}{{\mathbb{C}_{\text{cs}}}}
\newcommand{\ad}{\text{ad}}

\begin{document}

\maketitle

\begin{abstract}
In this paper, we study the geometric property (T) for discretized warped cones of an action on a compact Lie group  $M$ by its finitely generated subgroup. We show that if a subgroup $G$ is dense in $M$, then the associated discretized warped cone $\bigsqcup_n M\times \{t(n)\}$ does not have geometric property (T) for any sequence of positive numbers $\{t(n)\}_{n\in \mathbb{N}}$ converging to $\infty$. This result applies to certain ergodic actions of groups with property (T), for example, the action of $SO(d,\mathbb{Z}[\frac{1}{5}])$ on $SO(d)$ with $d\geq 5$.

As an application, we obtain new examples of expanders without geometric property (T), including certain superexpanders. 
\end{abstract}

\section{Introduction}

Let $M$ be a compact metric space, and $G$ a finitely generated group acting on $M$. Roe \cite{Roe2005Warped} introduced the concept of warped cone $\mathcal{O}_GM$ to construct exotic examples of metric spaces—for instance, spaces that lack Yu's property A or fail to admit a coarse embedding into Hilbert space. This construction provides a more flexible alternative to box spaces.
%, as a more flexible construction than box spaces.
It was shown that, under suitable conditions, certain dynamical properties of an action $G \curvearrowright M$ correspond to large-scale geometric properties of the warped cone $\mathcal{O}_G M$, such as amenability of action and Yu's property A (\cite{Roe2005Warped, Sawicki2021Straightening}), dynamical asymptotic dimension and asymptotic dimension (\cite{Sawicki2021Straightening}).

An expander is a sequence of finite, connected graphs with strong connectivity properties, playing a central role in many areas of mathematics (cf.~\cite{Lubotzky1994Discrete}). Vigolo \cite{Federico2019Measure} provided an explicit construction of expanders via warped cones arising from group actions with spectral gap. For a group $G$ with a finite generating set $S$, a probability measure preserving action on a probability measure space $(M,\mu)$ is said to have spectral gap if there exists $\delta>0$ such that for all $\xi\in L^2(M,\mu)$ with $\int_M \xi(x) d\mu(x)=0$, we have
$$\|\xi-\pi(s)\xi\|\geq \delta \|\xi\|$$
for some $s\in S$, where $\pi$ is the unitary representation $\pi:G\rightarrow U(L^2(M,\mu))$ induced by the measure preserving action. Instead of considering the entire warped cone, Vigolo \cite{Federico2019Measure} considered the coarse disjoint union of sections of $\mathcal{O}_GM$ at $t(n)$'s, where $\{t(n)\}_{n \in \mathbb{N}}$ is an increasing sequence of positive reals converging to infinity and we call it a warped system. Then he showed that an action has spectral gap if and only if for any (equivalently, some) sequence $\{t(n)\}_{n\in \mathbb{N}}$, the corresponding warped system is quasi-isometric to an expander.

In \cite{WillettYu2012higher, Willett2014Geometric}, Willett and Yu introduced a coarse invariant for metric spaces, known as geometric property (T), which serves as an obstruction to the surjectivity of the maximal coarse Baum--Connes assembly map. Roughly speaking, it concerns the existence of spectral gap for the Laplacian in the maximal Roe algebra associated with the space. So, for a coarse disjoint union of finite graphs, geometric property (T) is stronger than being an expander in the sense that the Laplacian has spectral gap in the Roe algebra. They showed that for a box space $X$ of a group $G$, $X$ has geometric property (T) if and only if $G$ has property (T). 

%\begin{qustn}Let $G$ be a finitely generated group, and $(M,d)$ a compact metric space with a $G$-action. When does the associated warped system have geometric property (T)?\end{qustn}

In \cite{Willett2014Geometric}, geometric property (T) was defined for discrete metric spaces, but to study it for warped systems, we adopt the non-discrete analogue formulated in \cite{Jeroen2021Geometric}. If the base space is a compact Riemannian manifold $M$ and the action is Lipschitz, then the associated warped cone can be approximated by a graph with uniformly bounded degree as shown in \cite[Proposition 1.10]{Roe2005Warped}.
Geometric property (T) in the sense of \cite{Jeroen2021Geometric} is equivalent to the geometric property (T) of its discretization in the sense of \cite{Willett2014Geometric}. In \cite[Question 11.2]{Jeroen2021Geometric}, Winkel asked the following question:
\begin{qustn}\label{Winkel Question}
If $G$ has property (T), $M$ is a compact Riemannian manifold and $G\curvearrowright M$ is an ergodic action by diffeomorphisms, then does the warped system have geometric property (T)? 
\end{qustn}
Margulis \cite{Margulis1980Some} and Sullivan \cite{Sullivan1981For} independently proved that for every $d\geq 5$, $SO(d)$ contains a finitely generated subgroup $G$ which is dense in $SO(d)$ and has property (T). The compact Lie group $SO(d)$ ($d\geq 5$) admits a $G$-action by the left multiplication $G\curvearrowright SO(d)$.
The following main result shows that this is a counterexample to Question \ref{Winkel Question}.

\begin{thm}\label{someseqIntro}
 \rm{
    Let $G$ be a finitely generated group which is dense in a compact Lie group $\overline{G}=M$ and $G\curvearrowright M$ an action induced by the left multiplication. In this case, for any sequence $\{t(n)\}_{n \in \mathbb{N}}$ converging to infinity, the warped system $\bigsqcup (X_{t(n)},\delta_G^{t(n)})$ associated to $G\curvearrowright M$ does not have geometric property (T).
}
\end{thm}

While numerous studies, such as those in \cite{Roe2005Warped, Drutu2019Kazhdan, Sawicki2019Warpedrigidity}, have explored the similarities between the large-scale geometric properties of box spaces and warped cones, the above result shows that their behaviors can differ significantly under maximal representations.

In the case when the action has spectral gap, the warped system $\bigsqcup X_{t(n)}$ is quasi-isometric to an expander \cite{Federico2019Measure}. By Theorem \ref{someseqIntro}, the warped systems associated with the left translation action of the groups constructed by Margulis and Sullivan $G$ ($d\geq 5$) on $SO(d)$ is an expander, yet it does not have geometric property (T). Moreover, in \cite{deLaatVigoro2019Superexpanders}, it was showed that the warped system corresponding to the example of Margulis $SO(d,\mathbb{Z}[\frac{1}{5}])\curvearrowright SO(d)$ ($d\geq 5$) is a superexpander (See \cite[Definition 2.2]{SawickiSuperexpandersandWarpedCones}). Therefore, Theorem \ref{someseqIntro} provides examples of superexpanders without geometric property (T).

This paper is organized as follows. In section 2, some basic concepts in coarse geometry as well as the definition of geometric property (T) and warped cones will be briefly recalled. In section 3, we discuss the spectrum of various Laplacians. In section 4, we prove our main theorem. In section 5, we remark that in Theorem \ref{someseqIntro}, if the base space $M$ is replaced by a Cantor set, warped systems can have geometric property (T).

\section{Preliminaries}
In this section, we first fix some notations, and then recall the definition of geometric property (T) for metric spaces and the definition of warped cones associated with group actions.

Through this paper, for any sets $X$ and any subsets $E$ of $X\times X$, we fix the following notations:
    \begin{enumerate}[(i)]
        \item for any positive integer $n$, the composition $E^{\circ n}$ is defined to be
        \begin{align*}
            E^{\circ n}:=\{(x,y)\in X\times X:&\exists (x_0,x_1,\cdots x_n)\in X^{n+1} \text{ such that }\\
            &x=x_0,y=x_n, (x_j,x_{j+1})\in E, \text{ }\forall j=0,1,\cdots n-1 \},
        \end{align*}

        \item and the inverse $E^{-1}$ is defined to be
        \begin{align*}
            E^{-1}:=\{(x,y)\in X\times X: (y,x)\in E\},
        \end{align*}

        \item for $x\in X$, the section $E_x$ of $E$ at $x$ is defined to be
        \begin{align*}
            E_x:=\{y\in X:(x,y)\in E\}.
        \end{align*}
    \end{enumerate}
The above operations on sets are used to study the coarse structure on a set by Roe \cite{Roe:Lecture-coarse-geometry}. To study the large scale geometry by the means of Laplacian, we need the following concepts of controlled sets. 

In this paper, we allow a distance function $d$ on a metric space $X$ to take the value $\infty$.

\begin{dfn}
\rm{
Let $(X,d)$ be a metric space. 
\begin{enumerate}[(i)]
    \item A controlled set is a subset $E$ of $X\times X$ such that
\begin{align*}
    \sup\{d(x,y):(x,y)\in E\}<\infty.
\end{align*}
It is called symmetric if $E=E^{-1}$.

\item A controlled set $E$ is called coarse generating if for any controlled set $E'$ there exists a positive integer $n$ such that $E'\subset E^{\circ n}$.
\end{enumerate}
    }
\end{dfn}

To study the large scale geometry of metric spaces, we need the following concepts of coarse equivalence.  
\begin{dfn}
\rm{
    Let $(X,d_X)$ and $(Y,d_Y)$ be metric spaces. A map $f:X\rightarrow Y$ is called a coarse equivalence if
    \begin{enumerate}[(i)]
        \item there exist two non-decreasing functions $\rho_{+},\rho_{-}:[0,\infty)\rightarrow[0,\infty)$ with $\lim_{t\to\infty}\rho_{\pm}(t)=\infty$, such that
        \begin{align*}
            \rho_{-}(d_X(x,y))\leq d_Y(f(x),f(y))\leq \rho_{+}(d_X(x,y))
        \end{align*}
        for all $x,y\in X$ and

        \item there exists $C\geq 0$ such that for any points $y\in Y$, there exists $x\in X$ such that $d_Y(y,f(x))\leq C$.
    \end{enumerate}
We say that two metric spaces $X$ and $Y$ are coarsely equivalent if there exists a coarse equivalence $f:X\rightarrow Y$. 
Furthermore, if the functions $\rho_{\pm}(t)$ have the form $Lt \pm C$, then the coarse equivalence $f$ is called a quasi-isometry, and the spaces $X$ and $Y$ are said to be quasi-isometric.

    }
\end{dfn}

It is known that geometric property (T) is invariant under coarse equivalences (cf.~\cite[Section 4]{Willett2014Geometric}). When studying metric spaces via operators, we need the following concepts.  
\begin{dfn}
\rm{
Let $(X,d)$ be a metric space with a measure $\mu$.
\begin{enumerate}[(i)]
    \item For $T\in B(L^2(X,\mu))$, we define its support to be
    \begin{align*}
     \text{supp}(T):=\{(x,y)&\in X\times X:\\
     &\chi_VT\chi_U\neq 0\text{ for all open neighborhoods }U\ni x, V\ni y \},
    \end{align*}
    where $\chi_U,\chi_V\in L^{\infty}(X,\mu)\subset B(L^2(X,\mu))$ are the characteristic functions of $U$ and $V$, respectively.
    \item $\ControlledSupport[X]$ is the subset of $B(L^2(X,\mu))$ consisting of all operators in whose support is controlled. When $X$ has bounded geometry with a measure $\mu$ (see Definition \ref{BoundedGeometry}), $\ControlledSupport[X]$ forms a $*$-subalgebra of $B(L^2(X,\mu))$. In this paper, this algebra is sometimes denoted by $\ControlledSupport[(X,d)]$ to emphasize the metric $d$ we are considering.
\end{enumerate}    
}
\end{dfn}

\subsection{Geometric property (T)}
In this subsection, we recall the definition and some properties of geometric property (T), especially for non-discrete spaces as in \cite{Jeroen2021Geometric}. For the non-discrete cases, the geometric property (T) is defined under the existence of a certain measure on the space. We will discuss the relationship between the measure and the coarse structures.

\begin{dfn}
\rm{
    Let $(X,d)$ be a metric space with a measure $\mu$.
    \begin{enumerate}[(i)]
        \item We say that $\mu$ is uniformly bounded if for any $r>0$, we have 
        $$\sup_{x \in X}\mu(B_r(x))<\infty,$$
       where $B_r(x)$ is the ball of radius $r$ centered at $x$.

       \item A controlled set $E\subset X\times X$ is said to be measurable if $\bigcup_{x\in U}E_x \subset X$ is measurable for all measurable $U\subset X$.

        \item A symmetric controlled set $E\subset X\times X$ is called $\mu$-gordo if it is measurable and $\mu(E_x)$ is bounded away from zero independently of $x\in X$.
    \end{enumerate}
    }
\end{dfn}

\begin{dfn}[Proposition 3.7 \cite{Jeroen2021Geometric}]\label{BoundedGeometry}
\rm{
    A metric space $(X,d)$ is said to have bounded geometry if there exist a uniformly bounded measure $\mu$ on $X$, and a symmetric controlled $\mu$-gordo set $E\subseteq X \times X$.
    }
\end{dfn}

For a metric space with bounded geometry, one can define a Laplacian associated with any measurable symmetric controlled set.

\begin{dfn}
\rm{
    Let $(X,d)$ be a metric space with a measure $\mu$ and $E\subset X\times X$ any measurable symmetric controlled set. Define the Laplacian $\Delta_E\in \ControlledSupport[X]$ by
    \begin{align*}
        \Delta_E(\xi)(x):=\int_{E_x}(\xi(x)-\xi(y))d\mu(y)
    \end{align*}
    for all $\xi\in L^2(X,\mu)$ and $x\in X$.
    }
\end{dfn}

The following definition of geometric property (T) was formulated in terms of spectral gap of Laplacians. We remark here that several equivalent definitions of geometric property (T) were given in (cf.~\cite[Definition 6.7, Definition 7.5, Proposition 7.6]{Jeroen2021Geometric}). 

\begin{dfn}[Definition 7.6, \cite{Jeroen2021Geometric}]
\rm{
Let $(X,d)$ be a metric space and let $\mu$ be a uniformly bounded measure for which a gordo set exists. We say that $(X,\mu)$ has geometric property (T) if there exists a measurable symmetric controlled set $E$ and a constant $\gamma>0$ such that for every unital $*$-representation $\rho:\ControlledSupport[X]\rightarrow B(\HH)$, we have
    \begin{align*}
        \sigma(\rho(\Delta_E))\subset \{0\}\cup [\gamma,\infty)
    \end{align*} and
    \begin{align*}
        \ker(\rho(\Delta_E))=\cap\{\ker(\rho(\Delta_F):F\subset X\times X {\text{ is measurable, symmetric and controlled}}\}.
    \end{align*}
    }
\end{dfn}

It is natural to expect that if a measurable symmetric controlled set $E\subset X\times X$ is "large enough", then the second condition of the above definition is automatic. This was formulated and proved in  \cite[Proposition 7.9]{Jeroen2021Geometric}.

\begin{prop}\label{KernelofLaplacians}
    \rm{
    Let $(X,d)$ be a metric space and $\mu$ be a uniformly bounded measure and E be a gordo set that generates the coarse structure on $X$. Let $E':=E^{\circ 3}$. Then for every unital $*$-homomorphism $\rho:\ControlledSupport[X]\rightarrow B(\HH)$, we have 
    \begin{align*}
        \ker(\rho(\Delta_{E'}))=\cap\{\ker(\rho(\Delta_F):F\subset X\times X {\text{ is measurable, symmetric and controlled}}\}.
    \end{align*}
    Moreover, $(X,d)$ has geometric property (T) if and only if there exists a constant $c_0>0$ such that the spectrum of $\rho(\Delta_{E'})$ is contained in $\{0\}\cup [c_0, \infty)$ for any unital $*$-homomorphism $\rho:\ControlledSupport[X]\rightarrow B(\HH)$. 
    }
\end{prop}

In the non-discrete case, it seems that the definition of geometric property (T) depends on the choice of measures. However, it was proven in \cite{Jeroen2021Geometric} that it does not dependent on the choice of measures, and it is also a coarse invariant (cf.~\cite[Theorem 8.6]{Jeroen2021Geometric}).

\begin{comment}
\begin{thm}[Theorem 8.6 \cite{Jeroen2021Geometric}]
\rm{
    Suppose that $X$ and $X'$ are coarsely equivalent spaces with bounded geometry and $\mu$ and $\mu'$ are uniformly bounded measures on $X$ and $X'$ respectively for which gordo sets exist. Then $(X,\mu)$ has geometric property (T)
    if and only if $(X',\mu')$ has geometric property (T).
    }
 \end{thm}
\end{comment}
To conclude this section, we formulate the following result on the positivity of the approximating operators for Laplacians in the next section. 
\begin{lem}\label{positive}
\rm{
    Let $\alpha:X\times X\rightarrow \mathbb{R}_+$ be a symmetric measurable bounded function such that $\int_X\alpha(x,y)d\mu(y)$ is uniformly bounded and
    $\int_X\int_X|\alpha(x,y)|^2d\mu(x)d\mu(y)<\infty$.
    Then the kernel operator 
    \begin{align*}
        T:L^2(X,\mu)\rightarrow L^2(X,\mu)
    \end{align*}
    given by $T\xi(x)=\int_X\alpha(x,y)(\xi(x)-\xi(y))d\mu(y)$ is positive in $B(L^2(X))$.
    }
\end{lem}

\begin{proof}
    For $\xi\in L^2(X)$, since $\alpha$ is symmetric, we have
    \begin{align*}
        \langle \xi, T\xi\rangle&=\int_X \int_{X}\alpha(x,y)(\xi(x)-\xi(y))\overline{\xi(x)}d\mu(y)d\mu(x)\\
        &=\frac{1}{2}\int_X \int_{X}\alpha(x,y)\left((\xi(x)-\xi(y))\overline{\xi(x)}+(\xi(y)-\xi(x))\overline{\xi(y)}\right)d\mu(y)d\mu(x)\\
        &=\frac{1}{2}\int_{X\times X} \alpha(x,y)|\xi(x)-\xi(y)|^2d(\mu\times \mu)(x,y)\geq 0.\qedhere
    \end{align*}
\end{proof}

\subsection{Warped systems}

In this subsection, we shall review the definitions of warped cones and warped systems. The concept of warped cones was introduced by Roe \cite{Roe2005Warped} to associate a coarse space with a group action on a compact metric space, and it encodes certain dynamical properties of the action. A warped system is a discretized version of a warped cone (cf.~\cite[Section 6]{Federico2019Measure}).

\begin{dfn}\label{warped system}\leavevmode
\rm{
\begin{enumerate}[(i)]
    \item Let $(X,d)$ be a proper metric space ane let $G$ be a group generated by a finite symmetric generating set $S \subset G$ acting by homeomorphisms on X. We denote by $\ell$ the associated length function on $G$. The warped distance $d_{G}(x,y)$ between two points $x,y\in X$ is defined to be
    \begin{align*}
       \displaystyle d_{G}(x,y):=\inf \sum \left(d(g_ix_i,x_{i+1})+\ell(g_i)\right),
    \end{align*}
    where the infimum is taken over all finite sequences $x=x_0,x_1,\cdots,x_N=y$ in $X$ and $g_0,g_1,\cdots, g_{N-1}$ in $G$.

    \item Let $(M,d)$ be any metric space and we fix an increasing sequence of positive numbers $\{t(n)\}_n$ with $\lim_{n\to \infty} t(n)=\infty$. For each $n$, let $(M,d^{t(n)})$ be ``the $t(n)$-times enlargement of $(M,d)$'', i.e. the distance $d^{t(n)}$ on $M$ defined by $d^{t(n)}(x,y):=t(n) d(x,y)$. If a finitely generated group $G$ acts on $M$, we can consider the warped distance of $d^{t(n)}$, which is denoted by $\delta_G^{t(n)}$. Then we obtain a sequence of metric spaces $(M,\delta_{G}^{t(n)})$ and a warped system is their disjoint union $\bigsqcup_n (M,\delta_{G}^{t(n)})$. Throughout this paper, by $\bigsqcup_n$, we mean that the distance between different components is taken to be infinite.
\end{enumerate}
    }
\end{dfn}

\begin{rmk}
\rm{
    If $(M,d)$ is a compact manifold, where $d$ is the distance determined by the Riemannian metric $g$ on $M$, Roe (in \cite{Roe2005Warped}) originally defined the warped cone to be the cone $[1,\infty)\times M$ with the warped distance $\delta_G$ of the cone distance $d_\text{cone}$ which is induced from the metric $t^2g+g_{\mathbb{R}}$ ($g_{\mathbb{R}}$ is the standard Euclidean metric). In  \cite[Lemma 6.5]{Federico2019Measure}, Vigolo proved that each $(M,\delta_G^{t(n)})$ is bi-Lipschitz equivalent to $(\{t(n)\}\times M, \delta_G)$ with a Lipschitz constant $C$ independent of $n$.
    }
\end{rmk}
The relationship between the properties of the action $G\curvearrowright M$ and the large scale geometric properties of the warped cone $\mathcal{O}_GM$ was investigated in \cite{Roe2005Warped, Sawicki2019Warpedrigidity, Sawicki2021Straightening}. In the following, we list a result necessary for this paper. 

\begin{lem}[Lemma 11.7, \cite{Jeroen2021Geometric}]\label{Coarse Generating}
\rm{
    Let $M$ be a compact Riemannian manifold with the Riemannian distance $d_M$, and $G$ a group with a finite symmetric generating set containing the identity. If $G$ acts on $M$ satisfying that each $g\in G$ is a Lipschitz homeomorphism on $M$, then for every $r>0$, the symmetric controlled set
    \begin{align*}
        E_r=\{(x,y)&\in (M,\delta^{t(n)}_G)\times (M,\delta^{t(n)}_G):\\
        &\exists~ x'\in M, s\in S \text{ s.t. } t(n)d_M(x,x')<r/2 \text{ and } t(n)d_M(sx',y)<r/2\}
    \end{align*}
    is a coarse generating set for the warped system $\bigsqcup(M,\delta^{t(n)}_G)$.
    }
\end{lem}

It is well known that a warped system is uniformly quasi-isometric to graphs with uniformly bounded degree under the setting of Lemma \ref{Coarse Generating} (see \cite[Proposition 1.10]{Roe2005Warped}). Therefore, we can construct examples of expanders when the group action has spectral gap.

\section{Laplacians}
In this section, we introduce several Laplacians related to warped cones and discuss the relationship between them. Our analysis is conducted in a more general setting than that assumed in the main theorem. Throughout the section, we assume that $M$ is an $m$-dimensional compact Riemannian manifold with the Riemannian distance $d_M$ and the measure $\mu$, and $G$ is a finitely generated group with a fixed symmetric generating set $S=S^{-1}\subset G$ containing the unit $e$. Let $G \curvearrowright M$ be an isometric, free and $\mu$-preserving action.

Let $\{t(n)\}_{n \in\mathbb{N}}$ be a sequence of positive numbers with $\lim_{n\to \infty}t(n)=\infty$, and let $X_{t(n)}:=M\times \{t(n)\}$ for each $n\in \mathbb{N}$. Denote by $\delta_{G}^{t(n)}$ the warped distance, and by $\mu_{t(n)}:=t(n)^m \mu$ on $X_{t(n)}$. We define the warped system corresponding to $\{t(n)\}_n$ by $X=\bigsqcup (X_{t(n)},\delta_G^{t(n)})$. For each $r>0$, we fix a symmetric coarse generating gordo set (defined in Lemma \ref{Coarse Generating}) 
\begin{align}
E_r=
\{(x,y)\in X_{t(n)}\times X_{t(n)}:\exists s\in S \text{ s.t. } t(n)d_M(sx,y)<r\}, \label{E_r}
\end{align}
since the action is isometric.
For each $x\in X$, we have 
$$(E_r)_x=\bigcup_{s\in S} B_{\frac{r}{t(n)}}(sx;d_M),$$
where $B_{\frac{r}{t(n)}}(x;d_M)$
 is the ball centered at $x$ with radius $\frac{r}{t(n)}$ with respect to the original distance $d_M$ of $M$. Since $G \curvearrowright M$ is free, there exists $r>0$ such that $s\cdot B_r(x;d_M)\cap s' \cdot B_r(x;d_M)=\phi$ for $x\in M$ and $s\neq s'\in S$ by the compactness of $M$. In the rest of the paper, we fix this $r$. Then we analyze the coarse Laplacian $\Delta_{E_r}$ associated with this coarse generating gordo set.  For $\xi\in L^2(X_{t(n)},\mu_{t(n)})$, we have that
\begin{align}
\begin{split}
    (\Delta_{E_r}\xi)(x)&=\int_{(E_r)_x} (\xi(x)-\xi(y) )d\mu_{t(n)}(y)\\
    &=\sum_{s\in S}\int_{s\cdot B_{\frac{r}{t_n}}(x;d_M)} (\xi(x)-\xi(y) )d\mu_{t(n)}(y)\\
    &=\sum_{s\in S} \int_{B_{\frac{r}{t_n}}(sx;d_M)} (\xi(x)-\xi(y) ) d\mu_{t(n)}(y)\\  
    &=\sum_{s\in S}t(n)^m \mu(B_{\frac{r}{t(n)}}(sx;d_M))\xi(x)-\sum_{s\in S}\int_{B_{\frac{r}{t(n)}}(sx;d_M)} \xi(y)d\mu_{t(n)}(y). \label{warped Laplacian}
    \end{split}
\end{align}
For each $n$ and the fixed $r$, we define the local Laplacian and the group Laplacian
\begin{align*}
   L_r=\bigoplus_n L_{r,n}&:\bigoplus L^2(X_{t(n)},\mu_{t(n)})\rightarrow \bigoplus L^2(X_{t(n)},\mu_{t(n)}),\\
   \Delta_G=\bigoplus_n \Delta_{G,n}&:\bigoplus L^2(X_{t(n)},\mu_{t(n)})\rightarrow \bigoplus L^2(X_{t(n)},\mu_{t(n)})
\end{align*}
by
\begin{align}\label{DefinitionofLocalandGroupLaplacians}
\begin{split}
(L_{r,n}\xi)(x)&=\int_{B_{\frac{r}{t(n)}}(x;d_M)}(\xi(x)-\xi(y))d\mu_{t(n)}(y)\\
&=t(n)^m\cdot\mu(B_{\frac{r}{t(n)}}(x;d_M))\cdot\xi(x)-\int_{B_{\frac{r}{t(n)}}(x;d_M)}\xi(y)d\mu_{t(n)}(y),\\
(\Delta_{G,n}\xi)(x)&=\sum_{s\in S}(\xi(x)-\xi(sx))=|S|\cdot \xi(x)-\sum_{s\in S}\xi(sx).
\end{split}
\end{align}
This local Laplacian $L_r$ is a Laplacian of $X$ associated to a coarse generating gordo set with respect to the non-warped distances $\bigsqcup (X_{t(n)},d^{t(n)})_{n\in \mathbb{N}}$. 

For each $n$, let us define a function on $M$ by
$$\phi_n(x):=t(n)^m\cdot\mu(B_{\frac{r}{t(n)}}(x;d_M))$$ 
for any $x\in M$.  
Define a function $\phi \in L^{\infty}(\bigsqcup (X_{t(n)},\mu_{t(n)}))$ by $\phi(x):=\phi_n(x)$ for any $x\in X_{t(n)}$. As a result of \eqref{warped Laplacian}, we have the following formula.

\begin{lem}\label{local and group laplacians}
\rm{
Let $\phi$ be the function defined as above, $L_r$ the local Laplacian, $\Delta_G$ the group Laplacian and $\Delta_{E_r}$ the coarse Laplacian. Then we have
\begin{enumerate}
 \item[(1)] $\Delta_{E_r}=|S|\phi-(|S|-\Delta_G)(\phi-L_{r})$;
 \item[(2)] the function $\phi$ and the local Laplacian $L_r$ is $G$-equivariant.
\end{enumerate}
    }
\end{lem}
\begin{proof}
For $\xi\in L^2(X_{t(n)},\mu_{t(n)})$,
\begin{align*}
    (\Delta_{E_r}\xi)(x)&=\sum_{s\in S}\phi_n(sx)\xi(x)-\sum_{s\in S}((\phi_n-L_{r,n})\xi)(sx)\\
    &=(|S|\phi_n\xi)(x)-\left((|S|-\Delta_G)(\phi_n-L_{r,n})\xi\right)(x).
\end{align*}
The $G$-equivariance of $\phi$ and $L_r$ follows from the assumption that the action is isometric and $\mu$-preserving.
\end{proof}

We analyze the Laplacian \(\Delta_{E_r}\) within a certain crossed product algebra, which is related to \(\ControlledSupport[\bigsqcup (X_{t(n)}, \delta_G^{t(n)})]\) via the $*$-homomorphism \(\Psi\) introduced in equation~\eqref{crossed product} in the next section. Since \(\Psi\) becomes an isomorphism modulo the ideal \(I = \bigoplus_n B(L^2(X_{t(n)}))\), we study the spectrum of the localized Laplacian \(L_r\) in the quotient by $I$. 
Recall from Definition \ref{warped system} that the non-warped distance $d^{t(n)}$ on $X_{t(n)}$ is $t(n)\cdot d$ and we denote the disjoint union by $Y=\bigsqcup(X_{t(n)},d^{t(n)})$.

\begin{prop}\label{QuotientofLocalLaplacian}
\rm{
   Let $\{t(n)\}_{n \ni \mathbb{N}}$ be a sequence of positive numbers with $\lim_{n \to \infty} t(n)=\infty$. Then the image of the local Laplacian $ L_r=\bigoplus_n L_{r,n}$ does not have spectral gap in the quotient $$\overline{\ControlledSupport[Y]}^{L^2}/\overline{I},$$ 
 where $I$ is the (algebraic) ideal $\bigoplus_n B(L^2(X_{t(n)}))$ of $\ControlledSupport[Y]$ and $\overline{\ControlledSupport[Y]}^{L^2}$ is the completion of $\ControlledSupport[Y]$ in $B(\bigoplus L^2(X_{t(n)}))$.
    }
\end{prop}
We prove this using the following Lemma \ref{hodge and local Laplacian} to compare the local Laplacian $L_r$ and the Hodge--de Rham Laplacian. Since the proof involves standard kernel estimates and would interrupt the flow of the main argument, we include it in the appendix.

To compare $L_{r,n}$ with $\exp(-s\Delta_M)$,
    in the proof of Proposition \ref{QuotientofLocalLaplacian} and in Lemma \ref{hodge and local Laplacian}, we regard the local Laplacians $L_{r,n}$ as an operator on $L^2(M,\mu)$ by the same formula as \eqref{DefinitionofLocalandGroupLaplacians}
    \begin{align*}
        (L_{r,n}\xi)(x)=\int_{B_{\frac{r}{t(n)}}(x;d_M)}(\xi(x)-\xi(y))t(n)^m d\mu(y).
    \end{align*}

\begin{lem}\label{hodge and local Laplacian}
\rm{
    Let $\Delta_M$ be the scalar Laplace--de Rham operator on $M$, i.e. the restriction of Hodge Laplacian to zero-forms (smooth functions). We consider the operator
    \begin{align*}
        \left(1-\exp{\left(-\frac{\Delta_M}{t(n)^2}\right)}\right)_n:\bigoplus L^2(M,\mu)\rightarrow \bigoplus L^2(M,\mu).
    \end{align*}
    Then for every $r>0$ there exist constants $C,D>0$, depending only on $r$ such that for every $\varepsilon>0$, there exists $R>0$ such that
    \begin{align*}
        0\leq L_{r,n}\leq C\left(1-\exp{\left(-\frac{\Delta_M}{t(n)^2}\right)}\right) \leq DL_{R,n}+\varepsilon
    \end{align*}
    for every $n$.
    }
\end{lem}

\begin{proof}[Proof of Proposition \ref{QuotientofLocalLaplacian}]
    We denote the quotient map by $q:B(\bigoplus L^2(M,\mu))\rightarrow B(\bigoplus L^2(M,\mu))/\overline{\bigoplus B(L^2(M,\mu))}$ and realize $B(\bigoplus L^2(M,\mu))/\overline{\bigoplus B(L^2(M,\mu))}$ as a concrete $C^*$-algebra in some $B(\HH)$. 
    Since by Lemma \ref{hodge and local Laplacian}, $L_r$ is dominated by $\left(1-\exp{\left(-\frac{\Delta_M}{t(n)^2}\right)}\right)_n$, to prove that $q(L_r)$ does not have spectral gap, it suffices to show that
    \begin{enumerate}[(i)]
        \item $\ker\left(q\left(1-\exp{\left(-\frac{\Delta_M}{t(n)^2}\right)}\right)_n\right)=\ker(q(L_r))$ and

        \item $q\left(1-\exp{\left(-\frac{\Delta_M}{t(n)^2}\right)}\right)_n$ does not have spectral gap.
    \end{enumerate}
    
    First, we prove (i). By Lemma \ref{hodge and local Laplacian}, it is clear that 
    \begin{align*}
         \ker\left(q\left(1-\exp{\left(-\frac{\Delta_M}{t(n)^2}\right)}\right)_n\right)\subset \ker(q(L_r)).
    \end{align*}
    On the other hand, if $\xi\in \ker{q(L_r)}\subset \HH$, then for all $\varepsilon>0$, we have
    \begin{align*}
       0\leq  \left\langle \xi,q\left(1-\exp{\left(-\frac{\Delta_M}{t(n)^2}\right)}\right)_n\xi \right\rangle \leq \frac{D}{C}\langle \xi, q(L_R)\xi \rangle+\frac{\varepsilon}{C} \|\xi\|^2=\frac{\varepsilon}{C} \|\xi\|^2
    \end{align*}
    because $\ker(q(L_r))=\ker(q(L_R))$ by Lemma \ref{KernelofLaplacians}. Since $\varepsilon$ is arbitrary, we have $\xi\in \ker\left(q\left(1-\exp{\left(-\frac{\Delta_M}{t(n)^2}\right)}\right)_n\right)$. This shows
    \begin{align*}
         \ker\left(q\left(1-\exp{\left(-\frac{\Delta_M}{t(n)^2}\right)}\right)_n\right)=\ker(q(L_r)).
    \end{align*}

    Next, we prove (ii). The scalar Laplace--de Rham operator $\Delta_M$ admits eigenvalues 
    \begin{align*}
        0=\lambda_0<\lambda_1\leq\lambda_2\leq\cdots\leq \lambda_i\leq
        \cdots \to \infty.
    \end{align*}
    So $\sigma_{L^2(X_t)}(1-\exp(-\frac{\Delta_M}{t^2}))=\left\{1-\exp(-\frac{\lambda_j}{t^2}):j=0,1,\cdots \right\}$.
    
    Let us denote $\sigma_{n,j}:=1-\exp{(-\frac{\lambda_j}{t(n)^2})}$. We show that for any $\varepsilon>0$, the closed interval $[\varepsilon,2\varepsilon]$ contains infinitely many $\frac{\lambda_j}{t(n)^2}$'s. Define a counting function $N:[0,\infty)\to \mathbb{N}$ of eigenvalues of $\Delta_M$ by
    \begin{align*}
        N(R):=\max\{j:\lambda_j\leq R\}.
    \end{align*}
By Weyl's Law (cf.~\cite[Corollary 2.43]{Berline-Getzler-Vergne:Heat=kernels-Dirac}), we know that 
$$\frac{N(R)}{R^{m/2}}\rightarrow C$$ 
for some constant $C$. Therefore, for every suitably large $t(n)$, we have
$N(2\varepsilon t(n)^2)-N(\varepsilon t(n)^2)>1$. So, there exists $j$ such that 
$$\varepsilon t(n)^2<\lambda_j\leq 2\varepsilon t(n)^2.$$   
Therefore, for any $\varepsilon>0$, there exists $0<\delta<\varepsilon$ which is an accumulation point of $\{\sigma_{n,j}\}$. It suffices to show that 
$$\delta\in \sigma\left(q\left(\left(1-\exp(-\frac{\Delta_M}{t(n)^2})\right)_n\right)\right)$$ 
in the quotient $B(\bigoplus L^2(M,\mu))/\overline{\bigoplus B(L^2(M,\mu))}$. If not, for $T_n:=1-\exp(-\frac{\Delta_M}{t(n)^2})-\delta$, there exists $(S_n') \in B(\bigoplus L^2(M,\mu))$ such that  $1_{X_{t(n)}}-T_nS_n'\rightarrow 0$ in norm. Therefore, for large enough $n$, we have $$\|1_{X_{t(n)}}-T_nS_n'\|<\frac{1}{2}.$$ 
As a result, $T_nS_n'$ has inverse whose norm is smaller than $2$. Then $S_n:=S_n'(T_nS_n')^{-1}$ is an inverse of $T_n$. So we have obtained  a sequence $(S_n)_n\in B(\bigoplus L^2(M,\mu))$ such that 
$$\left(1-\exp(-\frac{\Delta_M}{t(n)^2})-\delta\right)S_n=1_{X_{t(n)}}$$
for all but finitely many $n$'s. But for any $K>0$, there exists $n,j$ such that $|\sigma_{n,j}-\delta|<\frac{1}{K}$. On the eigenspace of $\Delta_M\in B(L^2(M,\mu))$ corresponding to $\lambda_j$, $S_n$ is bounded below by $K$. Since $K$ is arbitrary, the sequence $\{S_n\}_{n \in \mathbb{N}}$ can not be bounded, thus is it not an element in $(S_n)$ is not an element in $B\left(\bigoplus L^2(M,\mu)\right)$. This finishes the proof of (ii).
\end{proof}

\section{Proof of the main result}
In this section, we prove the following main result of this paper. 
\begin{thm}\label{subsequence}
    \rm{
    Let $G$ be a finitely generated group which is dense in a compact Lie group $\overline{G}=M$. For the action induced by the left multiplication $G\curvearrowright M$ and any sequence $\{t(n)\}_{n \in \mathbb{N}}$ converging to infinity, the warped system $\bigsqcup (X_{t(n)},\delta_G^{t(n)})$ does not have geometric property (T).
    
}
\end{thm}

We begin under the same setting as the previous section without assuming that $G$ is a dense subgroup of a compact Lie group $\overline{G}=M$ until Lemma \ref{HilberSchmidtSpectrum}. As well as the previous section, we denote by $X=\bigsqcup(X_{t(n)},\delta^{t(n)}_G)$ and $Y=\bigsqcup(X_{t(n)},d^{t(n)})$, the disjoint unions of $X_{t(n)}$ with warped or non-warped distance, respectively.
We use the fact from \cite[Lemma 11.8]{Jeroen2021Geometric} that there is a natural $*$-homomorphism
\begin{align}\label{crossed product}
    \Psi:\ControlledSupport[Y]\rtimes G \rightarrow \ControlledSupport[X].
   \end{align}
   This can be extended to a $*$-homomorphism between maximal completions:
   \begin{align}
    \Psi:\overline{\ControlledSupport[Y]}^{\max}\rtimes_{\max}G \rightarrow \overline{\ControlledSupport[X]}^{\max},
   \end{align}
   which fits into the commutative diagram
\begin{center}
\begin{tikzcd}
\overline{\ControlledSupport[Y]}^{\max}\rtimes_{\max}G \arrow[r, "\Psi"] \arrow[d, "q\otimes 1"]
& \overline{\ControlledSupport[X]}^{\max} \arrow[d, "q" ] \\
\left(\overline{\ControlledSupport[Y]}^{\max}/\overline{I}\right)\rtimes_{\max}G \arrow[r,"\cong"]
&\overline{\ControlledSupport[X]}^{\max}/\overline{I},
\end{tikzcd}
\end{center}
where $I$ is the algebraic direct sum $\bigoplus B(L^2(M,\mu_{t(n)}))$, the right vertical map is the quotient map by $\overline{I}$ and the left vertical map is the map induced by the $G$-equivariant quotient. The bottom horizontal map is an isomorphism, if $G\curvearrowright M$ is free. Let 
$$\Tilde{\Delta}_{E_r}:=|S|\phi-(\phi-L_r)\otimes(|S|-\Delta_G)\in \overline{\ControlledSupport[X]}^{\max}\rtimes_{\max}G.$$
To prove Theorem \ref{subsequence}, it suffices to show that $(q\otimes 1)(\Tilde{\Delta}_{E_r})$ does not have spectral gap in $$\left(\overline{\ControlledSupport[Y]}^{L^2}/\overline{I}\right)\rtimes_{\max}G\cong\left(\overline{\ControlledSupport[Y]}^{L^2}\rtimes_{\max}G\right)/(\overline{I}\rtimes_{\max}G).$$ Here, we omitted the map induced by the quotient map $\overline{\ControlledSupport[Y]}^{\max}/\overline{I}\twoheadrightarrow \overline{\ControlledSupport[Y]}^{L^2}/\overline{I}$. 

Now we construct a covariant system $(\pi,U,\HH)$ of the $C^*$-dynamical system $G\curvearrowright \overline{\ControlledSupport[Y]}^{L^2}$, where $G$ acts on $\overline{\ControlledSupport[Y]}^{L^2}$ by the adjoint which is denoted by $\ad$. 
Let $\omega_n$ be the trace on $B(L^2(X_{t(n)},\mu_{t(n)}))$ and we denote by
$$\HH^{(n)}:=\{T\in B(L^2(X_{t(n)},\mu_{t(n)}):\omega_n(T^*T)<\infty\}$$ the Hilbert space consists of Hilbert-Schmidt class operators.
For a kernel operator $k\in L^2(X_{t(n)}\times X_{t(n)},\mu_{t(n)}\times \mu_{t(n)})$, we have
$$\omega_n(k^*k)=\int_{M\times M}\overline{k(x,y)}k(x,y)d(\mu_{t(n)}\times\mu_{t(n)})(x,y).$$
Then $G$ acts on $\HH^{(n)}$ by conjugation, which is denoted by $U^{(n)}$ and $\overline{\ControlledSupport[Y]}^{L^2}$ is represented on $\HH^{(n)}$ by the left multiplication restricted to $L^2(X_{t(n)},\mu_{t(n)})$, denoted by $\pi^{(n)}$.

We show $(\pi^{(n)},U^{(n)},\HH^{(n)})$ is covariant as follows. Since $\pi^{(n)}$ is normal, it suffices to show that 
$$\pi^{(n)}(\ad(g)S_1)S_2=U_g^{(n)}\pi^{(n)}(S_1)U^{(n)}_{g^{-1}}S_2$$ for any $g\in G$ and any rank-one operators $S_1=\xi_1\otimes \eta_1^*\in B(L^2(X_{t(n)},\mu_{t(n)}))$ and $S_2=\xi_2\otimes \eta_2^*\in \HH^{(n)}$ ($\xi_1,\xi_2,\eta_1,\eta_2\in L^2(X_{t(n)},\mu_{t(n)})$). This follows from the computation
\begin{align*}
    \pi^{(n)}(\ad(g)S_1)S_2&=(g\xi_1\otimes g\eta_1^*)\circ (\xi_2\otimes \eta_2^*)=\langle g\eta_1,\xi_2\rangle(g\xi_1\otimes \eta_2^*)\\
    U_g^{(n)}\pi^{(n)}(S_1)U^{(n)}_{g^{-1}}S_2&=U_g^{(n)}\circ(\xi_1\otimes \eta_1^*)\circ (g^{-1
    }\xi_2\otimes g^{-1}\eta_2^*)=\langle \eta_1,g^{-1}\xi_2 \rangle U_g^{(n)} (\xi_1\otimes g^{-1}\eta_2^*)\\
    &=\langle \eta_1,g^{-1}\xi_2 \rangle (g\xi_1\otimes \eta_2^*).
\end{align*} 
We denote
\begin{align*}
    L^2_G(X_{t(n)}\times X_{t(n)}):=\left\{k\in L^2(X_{t(n)}\times X_{t(n)},\mu_{t(n)}\times \mu_{t(n)}):\begin{array}{ll}
         k(x,y)=k(gx,gy)  \\
         \forall  g\in G,\forall x,y\in M
    \end{array}\right\}.
\end{align*}
This space can be viewed as a closed subspace of $\HH^{(n)}$ isometrically.
Then, $\pi^{(n)}(\Delta_G)=0$ restricted on $L^2_G(X_{t(n)}\times X_{t(n)})$.
For each $k\in L^2_G(X_{t(n)}\times X_{t(n)})$, we can view it as a Hilbert-Schmidt class operator. Then $L_r\circ k$ is again a kernel function whose $(x,y)$-value is equal to
$$(L_r\circ k)(x,y)=\phi(x)k(x,y)-\int_{t(n)M}\chi_{B_{\frac{r}{t(n)}(x)}}(z)k(z,y)d\mu_{t(n)}(z).$$
Therefore, $\pi^{(n)}(L_r)$ can be restricted to $L^2_G(X_{t(n)}\times X_{t(n)})$.

Let $(\pi,U,\HH):=\bigoplus_n (\pi^{(n)},U^{(n)},\HH^{(n)})$ be the covariant system of $G\curvearrowright \overline{\ControlledSupport[Y]}^{L^2}$. Then we obtain a $G$-invariant subspace
$$\HH_G:=\bigoplus L^2_G(X_{t(n)}\times X_{t(n)}).$$ We denote  by $\Tilde{\pi}$ the representation of the crossed product $$\Tilde{\pi}:\overline{\ControlledSupport[Y]}^{L^2}\rtimes G \to B(\HH)$$  associated to the covariant system $(\pi,U,\HH)$.

\begin{lem}\label{HilberSchmidtSpectrum}
    \rm{
    If $M$ is a compact Lie group and $G$ is a finitely generated discrete subgroup of $M$ that is dense in $M$, then there exists a unitary
\begin{align*}
    W:L^2_G(X_{t(n)}\times X_{t(n)}) \rightarrow L^2(X_{t(n)},\mu_{t(n)})
\end{align*}
    such that 
    $$\pi^{(n)}(L_r)|_{L^2_G(X_{t(n)}\times X_{t(n)})}=W^*L_{r,n}W.$$
    Moreover, the restriction 
    $\pi(L_r)|_{\HH_G}$ does not have spectral gap in the quotient algebra 
    \begin{align*}
    B(\HH_G)/\overline{\bigoplus B({L^2_G(X_{t(n)}\times X_{t(n)})})}.    
    \end{align*}
    
    }
\end{lem}

\begin{proof}

Let $e\in M=\overline{G}$ be the identity element. Every $G$-invariant function $k\in L^2_G(X_{t(n)}\times X_{t(n)})$ is determined by the function $x\mapsto k(x,e)$ on $X_{t(n)}$, and it is an element in $L^2(X_{t(n)}, \mu_{t(n)})$. Moreover, this functions can be approximated by continuous functions $f$ on $X_{t(n)}$, each of which can be viewed as a continuous $G$-invariant function $k'$ on $X_{t(n)}\times X_{t(n)}$ and we have
\begin{align*}
    \|k-k'\|^2_{L^2_G(X_{t(n)}\times X_{t(n)})}=&\int_{M\times M} |k(x,y)-k'(x,y)|^2d\mu_{t(n)}(x)\times d\mu_{t(n)}(y)\\
    =&\int_M \int_M |k(z,e)-k'(z,e)|^2d\mu_{t(n)}(z)d\mu_{t(n)}(y)\\
    =&\mu_{t(n)}(X_{t(n)}) \|k(\cdot,e)-f\|_{L^2(X_{t(n)})}.
\end{align*}
Therefore, the set $C_G(X_{t(n)}\times X_{t(n)})$ of continuous $G$-invariant functions is dense in $L^2_G(X_{t(n)}\times X_{t(n)})$. 

Now, we define $W: L^2_G(X_{t(n)}\times X_{t(n)})  \to L^2(X_{t(n)},\mu_{t(n)})$ by the extension of the map
$$C_G(X_{t(n)}\times X_{t(n)})  \to C(X_{t(n)}), \quad k\mapsto\sqrt{\mu_{t(n)}(M)}k(\cdot,e),$$ which can be shown to be an isometry by the same computation as above and clearly $W$ is surjective. Therefore, $W$ is a unitary. 

For all $k,k'\in C_G(X_{t(n)}\times X_{t(n)})$, we have
    \begin{align*}
    &\omega_n\left(\left({k'}^*(\pi^{(n)}(L_r) k\right)\right)\\
    =&\iint_{(x,y)\in M\times M} \overline{k'(x,y)}\phi(x)k(x,y)d(\mu_{t(n)}\times\mu_{t(n)})(x,y)\\
    &-\iint_{(x,y)\in M\times M} \overline{k'(x,y)}\left(\int_{M}\chi_{B_{\frac{r}{t(n)}(x)}}(z)k(z,y)d\mu_{t(n)}(z)\right)d(\mu_{t(n)}\times\mu_{t(n)})(x,y)\\
    =&\int_M\left(\int_M\overline{k'(x,y)}\phi(x)k(x,y)d\mu_{t(n)}(x)\right)d\mu_{t(n)}(y)\\
    &-\int_{M}\left(\int_M \overline{k'(x,y)}\left(\int_{M}\chi_{B_{\frac{r}{t(n)}(x)}}(z)k(z,y)d\mu_{t(n)}(z)\right)d\mu_{t(n)}(x)\right)d\mu_{t(n)}(y).
\end{align*}
Since $G$ is dense in $M$ and the $y$-integrants are a $G$-invariant functions, the above is equal to
\begin{align*}
    &\mu_{t(n)}(M)\left(\int_M \overline{k'(x,e)}\phi(x)k(x,e)d\mu_{t(n)}(x)\right)\\
    &-\mu_{t(n)}(M)\left(\int_M \overline{k'(x,e)}\left(\int_{M}\chi_{B_{\frac{r}{t(n)}(x)}}(z)k(z,e)d\mu_{t(n)}(z)\right)d\mu_{t(n)}(x)\right)\\
   = &\left\langle \sqrt{\mu_{t(n)}(M)}k'(\cdot,e), (L_r)(\sqrt{\mu_{t(n)}(M)}k(\cdot,e)) \right\rangle_{L^2(M,\mu_{t(n)})}=\langle Wk',L_r(Wk) \rangle_{L^2(M,\mu_{t(n)})}
\end{align*}

The second statement follows from Lemma \ref{QuotientofLocalLaplacian}.
\end{proof}

Now, we are ready to prove the main result.    
\begin{proof}[Proof of Theorem \ref{subsequence}]
\rm{
Since there is a quotient 
\begin{align*}
\left(\overline{\ControlledSupport[Y]}^{L^2}\rtimes_{\max}G\right)/(\overline{I}\rtimes_{\max}G)\twoheadrightarrow \overline{\Tilde{\pi}\left(\ControlledSupport[Y]\rtimes G\right)}/\overline{\Tilde{\pi}(I\rtimes G)},
\end{align*}
it suffices that $\Tilde{\pi}(\Tilde{\Delta}_{E_r})$ does not have spectral gap in the quotient by $\bigoplus B(\HH^{(n)})$. Since the subspace $L^2_G(X_{t(n)}\times X_{t(n)})$ consists of $G$-invariant vectors, we have $\Tilde{\pi}(\Delta_G)=0$.
It follows that
$$\Tilde{\pi}(\Delta_{E_r})=|S|\cdot \Tilde{\pi}(L_r)$$ on $\HH_G.$ 
    By Lemma \ref{HilberSchmidtSpectrum}, this operator does not have spectral gap in the quotient by $\bigoplus B(L^2_G(X_{t(n)}\times X_{t(n)}))$. As a result, the operator $\Delta_{E_r}$ does not have spectral gap in $\overline{\ControlledSupport[X]}^{\max}$. This finishes the proof.
   }
\end{proof}

\section{Some Remarks}
In this section, we make two remarks on Theorem \ref{subsequence}. 

First, we show that if $M$ is a Cantor set and the group $G$ has property (T), then there is an isometric, free and measure preserving action on it such that the associated warped system has geometric property (T). 

Next, we discuss the Laplacian $\Delta_{E_r}$ on $L^2(X)$ instead of the maximal completion. If the action $G\curvearrowright M$ has spectral gap (especially if $G$ has property (T) and the action is ergodic, as is expected to give a warped cone with geometric property (T) in \cite[Question 11.2]{Jeroen2021Geometric}), we can easily show that $\Delta_{E_r}$ has spectral gap in $B(\bigoplus L^2(X_{t(n)},\mu_{t(n)}))$ by Lemma \ref{local and group laplacians}.

%In this remark, we discuss the Laplacian $\Delta_{E_r}$ on $L^2(X)$ instead of the maximal completion. If the action $G\curvearrowright M$ has spectral gap (especially if $G$ has property (T) and the action is ergodic, as is expected to give a warped cone with geometric property (T) in \cite[Question 11.2]{Jeroen2021Geometric}), we have an interesting comparison of the spectrums in reduced and maximal completion. 

Let us recall the definition of property (T).
\begin{dfn}
\rm{
    A discrete group $G$ with a finite generating set $S\subset G$ is said to have property (T) if there exists $\varepsilon>0$ such that for any unitary representation $(\pi,\HH)$ of $G$, and for any non-zero vector $\xi  \in \HH^{\perp}_G$, there exists $s\in S\setminus\{e\}$ such that 
    \begin{align*}
        \|\pi(s)\xi-\xi\|\geq\varepsilon\|\xi\|,
    \end{align*}
     where $\HH_{G}$ is the closed subspace consists of all $G$-fixed vectors and $\HH_{G}^{\perp}$ is the orthogonal complement of $\HH_{G}$.
    }
 \end{dfn}

 \begin{rmk}\label{TwoSidedGap}
 \rm{
     Let $S\subset G$ be a finite subset with $e\in S$ and $\pi:G\to \HH$ a unitary representation.
     Assume there exists $\varepsilon>0$ such that for any $\xi\in \HH^{\perp}_G$ there exists $s\in S_0$ such that $\|\pi(s_0)\xi-\xi\|\geq\varepsilon\|\xi\|$. Then there exists $\delta>0$ depending only on 
     $\varepsilon$ such that $$\left\|\sum_{s\in S}\pi(s)\xi\right\|\leq (|S|-\delta)\|\xi\|$$
      for any $\xi\in \HH^{\perp}_G$. 
      
      This can be seen using the uniform convexity of the Hilbert space to
      \begin{align*}
          \left\|\sum_{s\in S}\pi(s)\xi\right\|=\left\|\sum_{s\in S\setminus \{e,s_0\}}\pi(s)\xi\right\|+\|\pi(s_0)\xi\|+\|\xi\|\leq (|S|-2)+\|\pi(s_0)\xi\|+\|\xi\|.
      \end{align*}
     }
 \end{rmk}

 \begin{rmk}
 \rm{
     In contrast to our main theorem, if the base space is a Cantor set $C$, then there is an isometric measure preserving action by some groups such that the associated warped system has geometric property (T). Let $G$ be a finitely generated group with a sequence $\{N_i\}$ of decreasing normal subgroups with finite index. The Cantor set $C$ can be realized as an inverse limit $\varprojlim G/N_i$ of the canonical quotients $G/N_i\twoheadrightarrow G/N_{i+1}$. This set $C$, equipped with the natural $G$-action by translations, admits a $G$-invariant metric and measure. In  \cite[Corollary 7.7]{Sawicki2018Warpedprofinite}, Sawicki showed that there exists a sequence of level sets $\{t(n)\}_{n \in \mathbb{N}}$ such that the coarse disjoint union $\bigsqcup G/N_i$ is quasi isometric to the corresponding warped system $\bigsqcup X_{t(n)}$ for the action $G\curvearrowright \varprojlim G/N_i$. Therefore, combining with the result by Willett and Yu \cite[Theorem 7.3]{Willett2014Geometric}, if $G$ has property (T) then the warped system $\bigsqcup X_{t(n)}$ has geometric property (T). Moreover, if the intersection is trivial $\cap N_i=\{e\}$, then the converse of the above implication is also true. 
    }
 \end{rmk}

Next, we analyze the Laplacian $\Delta_{E_r}$ in $B(L^2(X))$.
\begin{rmk}
\rm{
If the action $G\curvearrowright M$ is free, isometric, measure preserving and has spectral gap, then  $\Delta_{E_r}$ has spectral gap in $B(L^2(X))$. 

By definition, if an action has spectral gap, then it is automatically ergodic. In particular, $\phi_n$ is constant on $X_{t(n)}$ since it is $G$-invariant.
Note that the spectrum of $\Delta_G$ is contained in $\{0\}\cup (\delta,2|S|-\delta)$ for some $\delta>0$ by Remark \ref{TwoSidedGap}. Since $\Delta_G$ and $L_r$ commute, and $\phi_n$ is constant on each $X_{t(n)}$, for a function $f_n(\lambda_1,\lambda_2):=\phi_n|S|-(|S|-\lambda_1)(\phi_n-\lambda_2)$,
 $\Delta_{E_r}|_{L^2(X_{t(n)})}$ admits the joint spectrum decomposition
    \begin{align*}
    \Delta_{E_r}|_{L^2(X_{t(n)})}&=\int_{\sigma(L_{r,n})}\int_{\sigma(\Delta_{G,n})} f_n(\lambda_1,\lambda_2) dE_{\Delta_{G,n}}(\lambda_1)dE_{L_{r,n}}({\lambda_2})\\
    &=\int_{[0,2\phi_n]}\int_{(\delta,2|S|-\delta)} f_n(\lambda_1,\lambda_2) dE_{\Delta_{G,n}}(\lambda_1)dE_{L_{r,n}}({\lambda_2})\\
    &+\int_{[0,2\phi_n]}\int_{\{0\}} f_n(0,\lambda_2) dE_{\Delta_{G,n}}(\lambda_1)dE_{L_{r,n}}({\lambda_2}).
    \end{align*}
    Since on $(\delta,2|S|-\delta)\times [0,2\phi_n]$ we have $f\geq \delta\phi_n$, as illustrated in Figure \ref{fig:lambda-plane}, $\forall\xi \in\ker(\Delta_{G,n})^{\perp}\subset L^2(X_{t(n)})$ we have $\|\Delta_{E_r}\xi\|\geq \delta\phi_n\|\xi\|$. 
    Moreover, by ergodicity, we have 
$\ker(\Delta_{G,n}) = \langle 1_{X_{t(n)}} \rangle 
\subset \ker(L_{r,n})$. It follows that 
$\Delta_{E_r}|_{\ker(\Delta_{G,n})} \equiv 0$ on $L^2(X_{t(n)})$.
Therefore, for any 
$\xi \in \ker(\Delta_{E_r}|_{L^2(X_{t(n)})})^{\perp}$, we obtain
\[
    \|\Delta_{E_r} \xi\| \geq \delta \phi_n \|\xi\|.
\]
Finally, note that $\phi_n$ is uniformly bounded below, since it converges 
to the volume of the ball of radius $r$ in the Euclidean space whose 
dimension is equal to that of $M$. 

Therefore, the Laplacian $\Delta_{E_r}$ has spectral gap in $B(L^2(X))$.
\begin{figure}[ht]
\centering
\begin{tikzpicture}[scale=1.0]
    \draw[->,>=stealth,semithick](-0.5,0)--(8,0)node[right]{$\lambda_1$};
    \draw[->,>=stealth,semithick](0,-0.5)--(0,4.75)node[left]{$\lambda_2$};
    \fill[red](1.5,0)--(4.5,0)--(4.5,2)--(1.5,2)--cycle;
    \draw[densely dotted](6,0)--(6,3);
    \draw[densely dotted](0,2)--(8,2);
    \draw[densely dotted](3,-2)--(3,4.75);
    \draw[densely dotted](-1,1)--(8,1);
    \draw(0,0)node[below left]{0};
    \draw(6,0)node[below]{$2|S|$};
    \draw(4.5,0)node[below]{$2|S|-\delta$};
    \draw(1.5,0)node[below]{$\delta$};
    \draw(0,2)node[left]{$2\phi_n$};
    \draw[blue](4.5,3)node[right]{$f_n=0$};
    \draw[blue](1.5,-1)node[left]{$f_n=0$};
    \draw[blue, domain=3.8:8] plot({\x},{1-3/(3-\x)});
    \draw[blue, domain=-1:2.0] plot({\x},{1-3/(3-\x)});
    \filldraw[red] (0,0) circle (4pt);
\end{tikzpicture}
\caption{Zeros of $f_n$ in $\sigma(\Delta_{G,n})\times \sigma(L_{r,n})$-plane}
\label{fig:lambda-plane}
\end{figure}
}
\end{rmk}

Our main result focuses on the cases of actions on a compact Lie group by a finitely generated dense subgroup, so it is natural to loosen the assumption on actions and ask the following questions.
\begin{qustn}
Is there a Lipschitz action $G\curvearrowright M$ by a finitely generated group $G$ on a compact manifold $M$ such that the associated warped system has geometric property (T)?
\end{qustn}

\appendix
\renewcommand{\thesection}{Appendix \Alph{section}}
\section*{Appendix: Proof of Lemma \ref{hodge and local Laplacian}}
\addcontentsline{toc}{section}{Appendix: Proof of Lemma \ref{hodge and local Laplacian}}

In this appendix, we provide a proof of Lemma \ref{hodge and local Laplacian}. The spectral relationship between the Hodge Laplacian and the local Laplacian was found in \cite[Lemma 10.3]{Jeroen2021Geometric}, but we need to relate the parameter $t(n)$ of the warped system and $s$ of the heat kernel.

%Let $(Y,d)$ be a metric space with a measure $\mu$. We denote by $\overline{\ControlledSupport[Y]}^{L^2}$ the completion of $\ControlledSupport[Y]$ under the norm in $B(L^2(Y))$, and by $\overline{\ControlledSupport[Y]}^{\max}$ the completion under the maximal norm
%\begin{align*}
  %  \|T\|_{\max}:=\sup \{\|\rho(T)\|_{B(\HH)}:\text{ } \rho:\ControlledSupport[Y]\rightarrow B(\HH) \text{ is a unital $*$-representation}\}.
%\end{align*}
%It was proved in \cite[Corollary 6.4]{Jeroen2021Geometric} that the maximal norm is finite in the case when $Y$ has bounded geometry (see Definition \ref{BoundedGeometry}).

%Let $\{t(n)\}_{n \ni \mathbb{N}}$ be a sequence of positive numbers. Recall that we view the local Laplacian $L_r$ as an element in $\ControlledSupport[\bigsqcup(X_{t(n)},d^{t(n)})]$, where $d^{t(n)}$ is the metric $t(n)\cdot d_M$ (see Definition 2.4).

\begin{proof}[Proof of Lemma \ref{hodge and local Laplacian}]
First, we prove that there exists $C>0$ such that 
    \begin{align*}
         L_{r,n}\leq C\left(1-\exp{\left(-\frac{\Delta_M}{t(n)^2}\right)}\right).
    \end{align*}
   We denote by $k_s$ the heat kernel of $\Delta_M$, so we have that
    $$\left(1-\exp(-s\Delta_M)\right)\xi=\int_M (\xi(x)-\xi(y)) k_s(x,y)d\mu(y)$$ 
    for $\xi\in L^2(M,\mu)$. 
We have the asymptotic estimate of the heat kernel by \cite[Theorem 7.15]{Roe1998Elliptic}
    \begin{equation}\label{HeatKernelEstimate}
        k_s(x,y)\sim \frac{1}{(4\pi s)^{m/2}}\exp{\left(-\frac{d(x,y)^2}{4s}\right)}(a_0(x,y)+a_1(x,y)s+\cdots)
    \end{equation}
    with $a_0,a_1,\cdots \in C^{\infty}(M\times M)$ satisfying $a_0(x,x)=1$ for all $x\in M$.
    So there exists $\ell\in \mathbb{N}$ and $C'>0$ such that for all $x,y\in M$ and $s\leq 1$, we have
    \begin{align*}
        \left|k_s(x,y)- p_s(x,y)(a_0(x,y)+a_1(x,y)s+\cdots a_{\ell}(x,y)s^{\ell})\right|\leq C' s,
    \end{align*}
    where we denote $p_s(x,y):=\frac{1}{(4\pi s)^{m/2}}\exp{\left(-\frac{d(x,y)^2}{4s}\right)}$.
    Fix $s_0$ small enough so that
    \begin{enumerate}[(a)]
        \item $C's_0\leq \frac{1}{(4\pi s_0)^{m/2}}\exp{\left(-\frac{r^2}{4}\right)}$;

        \item $\left|1-\left(a_0(x,y)+a_1(x,y)s+\cdots +a_{\ell}(x,y)s^{\ell}\right)\right|\leq \frac{1}{3}$ for every positive $s\leq s_0$.
    \end{enumerate}
    Such an $s_0$ exists because $M$ is compact and $a_i$ are continuous.
    For every $s\leq s_0$ if $d(x,y)<\sqrt{s}r$, then we have
    \begin{align*}
        &\left|p_s(x,y)-k_s(x,y)\right|\\
        \leq&\left|p_s(x,y)-p_s(x,y)(a_0(x,y)+a_1(x,y)s+\cdots a_{\ell}(x,y)s^{\ell})\right|\\
       & \quad\quad+\left|p_s(x,y)\left(a_0(x,y)+a_1(x,y)s+\cdots a_{\ell}(x,y)s^{\ell}\right)-k_s(x,y)\right|\\
       \leq & \frac{2}{3} p_s(x,y)
    \end{align*}
    and so $\frac{1}{3}\frac{1}{(4\pi s)^{m/2}}\exp{\left(-\frac{r^2}{4}\right)}\leq \frac{1}{3}p_s(x,y)\leq k_s(x,y) $.
    Therefore, since $k_s$ is positive everywhere we have 
    \begin{align*}
    s^{m/2}\chi_{B_{\sqrt{s}r}(x;d_M)}(y)\leq 3(4\pi)^{m/2}\exp{\left(\frac{r^2}{4}\right)}k_s(x,y).
    \end{align*}
    Applying this to Lemma \ref{positive}, we have 
    \begin{align}\label{DifferentialandLocalLaplacian}
        0\leq L_{r,n}\leq C\left(1-\exp{\left(-\frac{\Delta_M}{t(n)^2}\right)}\right)
    \end{align}
    for $C=3(4\pi)^{m/2}\exp{\left(\frac{r^2}{4}\right)}$.
    
    Next, we prove that there exists $D>0$ such that for any $\varepsilon>0$, there exists $R$ such that
    \begin{align}\label{second inequality}
        1-\exp{\left(-\frac{\Delta_M}{t(n)^2}\right)}\leq DL_{R,n}+\varepsilon.
    \end{align}
     For $s>0$ and $R>0$, we define a function $k_{s}^{(R)}:M\times M\to R$ by
    \begin{align*}
        {k}^{(R)}_{s}(x,y)&=\left\{
        \begin{array}{ll}
            k_s(x,y) & d(x,y)<\sqrt{s}R  \\
             0 & d(x,y)\geq \sqrt{s}R,
        \end{array}
        \right.
    \end{align*}
    and the corresponding Laplacian
    \begin{align*}
        \Delta_{k}^{(R)}=\bigoplus_n \Delta_{k,t(n)}^{(R)}: \bigoplus L^2(M,\mu)\rightarrow \bigoplus L^2(M,\mu)
    \end{align*}
    is defined by 
    \begin{align*}
        \left(\Delta_{k,t(n)}^{(R)}\xi\right)(x)&=\int_{M}{k}^{(R)}_{\frac{1}{t(n)^2}}(x,y)(\xi(x)-\xi(y)) d\mu(y).
    \end{align*}
    We show that for all $\varepsilon>0$, there exists $R>0$ such that 
    \begin{align}\label{deRhamleqlocal}
        0\leq 1-\exp{\left(-\frac{\Delta_M}{t(n)^2}\right)}\leq \Delta_k^{(R)}+\varepsilon.
    \end{align}
    Note that by $\eqref{HeatKernelEstimate}$, there exists a constant $D>0$ such that 
    \begin{align*}
        0\leq k_s(x,y)\leq Dp_s(x,y)
    \end{align*}
    for $s\leq \frac{1}{t(1)^2}$. Now for $\xi\in L^2(M,\mu)$, by the symmetry of heat kernel, we have
    \begin{align*}
        &\left\langle\left(1-\exp{\left(-\frac{\Delta_M}{t(n)^2}\right)}-\Delta_{k,t(n)}^{(R)}\right)\xi,\xi\right\rangle\\
        =& \int_M\int_{d(y,x)\geq \frac{R}{t(n)}} {k}_{\frac{1}{t(n)^2}}(x,y)(\xi(x)-\xi(y))\overline{\xi(x)} d\mu(y)d\mu(x)\\
        =&\frac{1}{2} \int_M\int_{d(y,x)\geq \frac{R}{t(n)}} {k}_{\frac{1}{t(n)^2}}(x,y)\left((\xi(x)-\xi(y))\overline{\xi(x)}+(\xi(y)-\xi(x))\overline{\xi(y)} \right)d\mu(y)d\mu(x)\\
        =& \frac{1}{2}\int_M\int_{d(y,x)\geq \frac{R}{t(n)}} {k}_{\frac{1}{t(n)^2}}(x,y) |\xi(x)-\xi(y)|^2d\mu(y)d\mu(x)\\
        \leq&\frac{D}{2}\int_M \int_{d(y,x)\geq \frac{R}{t(n)}} {p}_{\frac{1}{t(n)^2}}(x,y)|\xi(x)-\xi(y)|^2d\mu(y)d\mu(x)\\
        \leq &\frac{D}{2}\int_{\left\{(x,y)\in M\times M:\substack{|\xi(x)|\geq |\xi(y)|\\d(x,y)\geq R/t(n)}\right\}} {p}_{\frac{1}{t(n)^2}}(x,y) (2|\xi(x)|)^2 d(\mu\times\mu)(x,y)\\
        &+ \frac{D}{2}\int_{\left\{(x,y)\in M:\substack{|\xi(x)|\leq |\xi(y)|\\d(x,y)\geq R/t(n)}\right\}}  {p}_{\frac{1}{t(n)^2}}(x,y) (2|\xi(y)|)^2 d(\mu\times\mu)(x,y)\\
        \leq& 4D\int_M \int_{d(y,x)\geq \frac{R}{t(n)}}{p}_{\frac{1}{t(n)^2}}(x,y) |\xi(x)|^2 d\mu(y)d\mu(x)\\
        \leq&4D\sup_{x\in M} \left\{\int_{d(y,x)\geq \frac{R}{t(n)}} \frac{t(n)^{m}}{(4\pi)^{m/2}}\exp{\left(-\frac{t(n)^2d(x,y)^2}{4}\right)} d\mu(y) \right\}\|\xi\|_{L^2(M,\mu)}^2
    \end{align*}
    but the coefficient of $\|\xi\|_{L^2(M,\mu)}^2$ converges to $0$ when $R$ goes to infinity independently of $n$ by the change of variable in the integration. Now \eqref{deRhamleqlocal} is proved.
    
    Since $k_s^{(R)}(x,y)\leq  \frac{D}{(4\pi s)^{m/2}}\chi_{B_{\sqrt{s}R}(x;d_M)}(y)$, by Lemma \ref{positive} we have $\Delta_k^{(R)}\leq D L_{R,n}$. Combining this with \eqref{deRhamleqlocal}, we obtain the desired estimate \eqref{second inequality}.
\end{proof}

\section*{Acknowledgements} 

We would like to thank Prof. Guoliang Yu for his comments and discussions on this topic. We also appreciate Prof. Hanfeng Li for his comments.

\bibliographystyle{alpha}
\bibliography{main}
\end{document}